\documentclass[a4paper, 11pt, reqno]{amsart}
\usepackage{amssymb,amsmath,latexsym,amsthm,mathrsfs}

\theoremstyle{thmit} % Numbered and Italic
\newtheorem{thm}{Theorem}[section]
\newtheorem{lem}[thm]{Lemma}

\newtheorem{prop}[thm]{Proposition}

\theoremstyle{thmrm} % Numbered and Roman

\newtheorem*{oldproof}{Proof}
\renewenvironment{proof}[1][{}]{\begin{oldproof}[#1]}{\qed\end{oldproof}}

\theoremstyle{definition}
\newtheorem{remark}[thm]{Remark}

\newtheorem{example}[thm]{Example}
\newtheorem{assumption}[thm]{Assumption}

\numberwithin{equation}{section}

\newcommand{\abs}[1]{\lvert#1\rvert}
\DeclareMathOperator*{\Res}{Res}

\def\n{\mathbb{N}}
\def\nn{\mathbb{N}_{0}}

\def\p{{\mathbb{P}}}

\def\r{{\mathbb{R}}}
\def\c{{\mathbb{C}}}
\def\od{{\rm ord}}
\def\la{{\Lambda}}

\allowdisplaybreaks

\begin{document} 

\title{Double Dirichlet series associated with arithmetic functions}
\author{Kohji Matsumoto}
\address{K. Matsumoto\\ Graduate School of Mathematics, \\
Nagoya University, \\
Chikusa-ku, Nagoya 464-8602, Japan}
\email{kohjimat@math.nagoya-u.ac.jp}
\author{Akihiko Nawashiro}
\address{A. Nawashiro\\ Dai Nippon Printing Co., Ltd., \\
1-1-1, Ichigaya-Kagacho, \\
Shinjuku-ku, Tokyo \\
162-8001, Japan}
\email{nawashiro-a@mail.dnp.co.jp}
\author{Hirofumi Tsumura}
\address{H. Tsumura\\ Department of Mathematical Sciences, \\
Tokyo Metropolitan University,\\
1-1, Minami-Ohsawa, Hachiouji, Tokyo \\
192-0397, Japan}
\email{tsumura@tmu.ac.jp}



\thanks{This work was supported by Japan Society for the Promotion of Science, Grant-in-Aid for Scientific Research No. 25287002 (K. Matsumoto) and No. 15K04788 (H. Tsumura).}

\maketitle

\begin{abstract}
We consider double Dirichlet series associated with arithmetic functions such as the von Mangoldt function, the M\"obius function, and so on. We show analytic continuations of them by use of the Mellin-Barnes integral, and determine the location of singularities. Furthermore we observe their reverse values at non-positive integer points. 
\end{abstract}

\bigskip
\section{Introduction}\label{sec-1}

Let $\n$ be the set of natural numbers, $\nn :=\n \cup \{ 0\}$, $\p$ the set of prime numbers, %$\z$ the ring of rational integers, $\q$ the field of rational numbers, 
$\r$ the field of real numbers, $\c$ the field of complex numbers and $i:=\sqrt {-1}$.

For any arithmetic function $\alpha\,:\n \to \c$, let
\begin{equation}\label{def:DS}
\Phi (s;\alpha )=\sum _{n=1}^\infty \frac {\alpha (n)}{n^s}
\end{equation}
be the Dirichlet series associated with $\alpha$. 
It is important to show the analytic continuation of $\Phi (s;\alpha )$ for various $\alpha$'s. 
As its multiple version, we consider the multiple Dirichlet series
\begin{align}\label{def:mDS}
\Phi _r &(s_1,\cdots ,s_r;\alpha _1,\cdots ,\alpha _r)\\
&=\sum _{m_1=1}^\infty \sum _{m_2=1}^\infty \cdots \sum _{m_r=1}^\infty \frac {\alpha _1 (m_1)\alpha _2 (m_2)\cdots \alpha _r(m_r)}{m_1^{s_1}(m_1+m_2)^{s_2}\cdots (m_1+\cdots +m_r)^{s_r}},\nonumber
\end{align}
where $\alpha_k$ ($1\leq k\leq r$) are arithmetic functions. The typical example of \eqref{def:mDS} is the multiple zeta function of Euler-Zagier type defined by
\begin{equation}
\zeta _r (s_1,\cdots ,s_r)=\sum _{m_1=1}^\infty \sum _{m_2=1}^\infty \cdots \sum _{m_r=1}^\infty m_1^{-s_1}(m_1+m_2)^{-s_2}\cdots (m_1+\cdots +m_r)^{-s_r},  \label{EZ}
\end{equation}
which has been studied extensively in these two decades. In fact, 
the meromorphic continuation of \eqref{EZ} to the whole space $\c ^r$ has been proved by Akiyama, Egami and Tanigawa \cite{AET}, Zhao \cite{Zhao}, the first-named author \cite{Matsumoto-Nagoya,Matsumoto-JNT}, and so on.
In particular, the method in \cite{Matsumoto-Nagoya,Matsumoto-JNT} is to make use of the Mellin-Barnes integral formula (see, for example, \cite[Section 14.51, p.289, Corollary]{WW}):
\begin{equation}\label{MB}
(1+\lambda )^{-s}=\frac 1{2\pi i}\int _{(c)}\frac {\Gamma (s+z)\Gamma (-z)}{\Gamma (s)}\lambda ^{z}dz,
\end{equation}
where $s,\lambda \in \c$ with $\lambda \neq 0,\abs {{\rm arg}\lambda}<\pi,\ \Re s>0$, $c\in \r$ with $-\Re s<c<0$, and the path of integration is the vertical line form $c-i\infty$ to $c+i \infty$. This method is based on Katsurada's results \cite{Katsu1, Katsu2}. 

By applying this method, Tanigawa and the first-named author proved the following general result.
(We also mention a paper of de la Bret{\`e}che \cite{dlB} who treated similar multiple
series.)
Let $\mathcal{A}$ be the set of arithmetic functions satisfying the following three 
conditions: If $\alpha\in\mathcal{A}$, then
\begin{itemize}
\item[(I)] $\Phi (s;\alpha)$ is absolutely convergent for $\Re s>\delta=\delta(\alpha)(>0)$;
\item[(II)] $\Phi (s;\alpha)$ can be continued meromorphically to the whole plane $\c$, holomorphic except for a possible pole (of order at most 1) at $s=\delta$; 
\item[(III)] in any fixed strip $\sigma _1\leq \sigma \leq \sigma _2$, $\Phi (\sigma +it;\alpha)=O(\abs t^A)$ holds as $\abs t\to \infty$, where $A$ is a non-negative constant.
\end{itemize}

\begin{thm}[Matsumoto-Tanigawa{\cite[Theorem 1]{MatsumotoTanigawa}}]\label{MT-03} 
If arithmetic functions $\alpha_1,\ldots,\alpha_r$ belong to $\mathcal{A}$, then
$\Phi _r (s_1,\cdots ,s_r;\alpha _1,\cdots ,\alpha _r)$ can be continued meromorphically to the whole space $\c ^r$, and location of its possible singularities can be described explicitly. In particular, if all $\Phi (s;\alpha_k)$'s are entire, then $\Phi _r (s_1,\cdots ,s_r;\alpha _1,\cdots ,\alpha _r)$ is also entire.
\end{thm}

\ 
They applied this result to multiple Dirichlet $L$-functions of Euler-Zagier type and multiple automorphic $L$-functions. 

On the other hand, as an example which is outside of Theorem \ref{MT-03}, Egami and the first-named author \cite{MatsumotoEgami} considered the double series associated with the von Mangoldt function. Let $\zeta(s)$ be the Riemann zeta function and denote by 
$\{\rho_{n}\}_{n\geq 1}$ the non-trivial zeros of $\zeta(s)$ numbered by the size of absolute values of their imaginary parts. 
Let 
$\Lambda$ be the von Mangoldt function 
defined by
\begin{eqnarray*}
\Lambda (n)=\left\{ \begin{array}{ll}
\log p & (n=p^m\ \text{\rm for}\ p\in \mathbb{P},\ m\in \mathbb{N}) \\
0 & ({\rm otherwise}) \\
\end{array} \right. .
\end{eqnarray*}
Then it is well-known that
\begin{equation*}
\Phi (s;\la )=\sum _{n=1}^\infty \frac {\la (n)}{n^s}=-\frac {\zeta '(s)}{\zeta (s)}
\end{equation*} 
(see \cite[$\S$ 1.1]{Titch}). We denote this function by $M(s)$. 
We see that $M(s)$ has poles at $s=1$, $s=-2m\ (m\in \mathbb{N})$ and $s=\rho_l$ $(l\in \mathbb{N})$, hence does not satisfy the assumption (II).    That is, 
$\Lambda\notin\mathcal{A}$.
As a double version of $M(s)$, they considered 
\begin{eqnarray*}
{\mathcal M}_2(s)&=&\sum _{m_1=1}^\infty \sum _{m_2=1}^\infty \frac {\Lambda (m_1)\Lambda (m_2)}{(m_1+m_2)^s}\ (=\Phi _2(0,s;\Lambda ,\Lambda ))
\end{eqnarray*}
which can be written as $\sum _{n=1}^\infty G_2(n)n^{-s}$,
where
\begin{equation*}
G_2(n)=\sum _{m_1+m_2=n}\Lambda (m_1)\Lambda (m_2).
\end{equation*}
It should be emphasized that ${\mathcal M}_2(s)$ is closely connected to the famous Goldbach conjecture which implies that $G_2(n)>0$ for all even $n\geq 4$. 
The study of ${\mathcal M}_2(s)$ will be useful to understand 
the behaviour of $G_2(n)$. However, we cannot apply Theorem \ref{MT-03} to ${\mathcal M}_2(s)=\Phi _2(0,s;\Lambda ,\Lambda )$. 
 
%From this motivation, 
They also showed that the line $\Re s=1$ is the natural boundary of ${\mathcal M}_2(s)$ under some plausible assumptions, hence ${\mathcal M}_2(s)$ cannot be continued meromorphically to the whole complex plane $\c$ (see \cite[Theorem 2.2]{MatsumotoEgami}).

In the present paper we first consider another type of double version of $M(s)$ defined by 
\begin{equation}\label{def:phi(1,l)}
\Phi _2 (s_1,s_2;1,\Lambda )=\sum _{m_1=1}^\infty \sum _{m_2=1}^\infty \frac {\Lambda (m_2)}{m_1^{s_1} (m_1+m_2)^{s_2}}.
\end{equation}
The right-hand side of (\ref{def:phi(1,l)}) is absolutely convergent for $\Re s_2>1,\Re (s_1+s_2)>2$.
However, similar to ${\mathcal M}_2(s)$, we cannot apply Theorem \ref{MT-03} to $\Phi _2 (s_1,s_2;1,\Lambda )$. The first main aim of this paper is to prove that $\Phi _2(s_1,s_2;1,\la )$ can be continued meromorphically to the whole space $\c ^2$ (see Theorem \ref{state1}), and location of its singularities can be described explicitly (see Theorem \ref{C-2-1}), 

We here remark that 
\begin{equation}
\Phi _2 (0,s;1,\Lambda )=\sum_{n=1}^\infty \frac{\psi(n-1)}{n^s}, \label{Cheby}
\end{equation}
where $\psi(x)$ is the Chebyshev $\psi$-function defined by 
$\psi(x)=\sum_{1\leq n\leq x}\Lambda(n)$ (see \cite[Chap.\,7,\,\S 2]{Stein}). 
It is well-known that the prime number theorem comes from $\psi(x)\sim x$ $(x\to \infty)$, and also from $\widehat{\psi}(x)\sim \frac{1}{2}x^2$ $(x\to \infty)$, where
$$\widehat{\psi}(x)=\sum_{2\leq n\leq x}\psi(n-1)\left(=\int_{1}^{x-1} \psi(u)du\right)$$
(see \cite[Chap.\,7,\,Propositions 2.1 and 2.2]{Stein}). 
Note that applying Perron's formula to \eqref{Cheby}, we have 
$$\widehat{\psi}(x)=\frac{1}{2\pi i}\int_{c-iT}^{c+iT}\ \Phi_2(0,s;1,\Lambda)\frac{x^s}{s}ds+(\text{error term})\quad (c>2).$$
From this viewpoint as well, it seems important to study $\Phi_2(s_1,s_2;1,\Lambda)$.

Secondly, for any $\Phi (s;\alpha)$ with $\alpha\in\mathcal{A}$, 
we define $\widetilde{\alpha}:\,\n \to \c$ by
\begin{equation}\label{def_tilde_alpha}
\frac{\Phi (s;\alpha)}{\zeta(s)}=\sum _{n=1}^\infty \frac {\widetilde{\alpha}(n)}{n^s}\left(=:\Phi(s;\widetilde{\alpha})\right), 
\end{equation}
and consider the double series
\begin{equation}\label{def_double_tilde}
\Phi _2 (s_1,s_2;1,\widetilde{\alpha} )=\sum _{m_1=1}^\infty \sum _{m_2=1}^\infty \frac {\widetilde{\alpha} (m_2)}{m_1^{s_1} (m_1+m_2)^{s_2}}.
\end{equation}
The second main aim of this paper is to prove that $\Phi _2(s_1,s_2;1,\widetilde{\alpha}  )$ can be continued meromorphically to the whole space $\c ^2$ (see Theorem \ref{Th-4-2}). Note that we can apply this result to the cases when $\widetilde{\alpha}$ is the M\"obius function $\mu$, the Euler totient function $\phi$, and so on (see Section \ref{sec-6}). 
We further calculate reverse values of $\Phi _2 (s_1,s_2;1,\Lambda )$ at points on the sets of singularity (see Propositions \ref{spv1} and \ref{spv2} and Example \ref{Example-4-1}), and also those of $\Phi _2 (s_1,s_2;1,\mu )$ (see Example \ref{Exam-4-3}).

If $\widetilde{\alpha}$ is an arithmetic function for which $\Phi(s;\widetilde{\alpha})$
has only finitely many poles,  then the double series of the form \eqref{def_double_tilde}
has been studied by Choie and the first-named author \cite{C-M}.   However 
$\Phi(s;\widetilde{\alpha})$
defined by \eqref{def_tilde_alpha} obviously has infinitely many poles, so is outside of
the study in \cite{C-M}.

Also, since $\zeta'(s)$ has a double pole at $s=1$, it does not satisfy the assumption (II).
Therefore we will consider the case of the von Mangoldt function separately in Section
\ref{sec-2}.    The other reason of this separate treatment is that we need not assume
Assumption \ref{Ass-2} in Section \ref{sec-2}.

The authors believe that it is not difficult to extend the results in the present paper
to the case when $\Phi(s;\widetilde{\alpha})$ has a pole of higher order at $s=\delta$,
or even to the case when $\Phi(s;\widetilde{\alpha})$ has other finitely many poles.

\ 

%%%%%%%%%%%%%%%%%%%%%%%%%%%%%%%%%%%%%%%%%%%%%%%%%%%%%%%%%
\section{The double series $ \Phi _2 (s_1,s_2;1,\Lambda )$}\label{sec-2}
%%%%%%%%%%%%%%%%%%%%%%%%%%%%%%%%%%%%%%%%%%%%%%%%%%%%%%%%%

In this section, we consider analytic properties of $\Phi _2 (s_1,s_2;1,\Lambda )$. We first define $a_k$ and $b_l$ by 
\begin{align*}
& a_k=\frac 1{2\pi i}\int _{|\xi +k|=\frac 1 2} \frac {M(\xi )}{\xi +k} d\xi\quad (k\in \n;\ k:\text{even}), \\ 
& b_l=\frac 1{2\pi i}\int _{|\xi +l|=\frac 1 2} \frac {\Gamma (\xi )}{\xi +l} d\xi\qquad (l\in \n_0), 
\end{align*}
namely $a_k$ and $b_l$ are constant terms of Laurent series of $M(s)=-\zeta'(s)/\zeta(s)$ at $s=-k$ and of $\Gamma(s)$ at $s=-l$, respectively. We can also express that
\begin{align}
a_k&=\lim_{s\to -k}\ \frac{d}{ds}(s+k)M(s)=-\frac{\zeta''(-k)}{2\zeta'(-k)}\quad (\text{$k$:even}),\label{def-ak}\\
b_l&=\lim_{s\to -l}\ \frac{d}{ds}(s+l)\Gamma(s)=\frac{(-1)^l}{l!}\left(\sum_{j=1}^{l}\frac{1}{j}-\gamma\right), \label{def-bk}
\end{align}
where $\gamma$ is the Euler constant. 

\begin{thm}\label{state1}$\Phi _2(s_1,s_2;1,\la )$ can be continued meromorphically to the
whole space $\c ^2$ by the following expression:
\begin{align}\label{mainthm}
&\Phi _2 (s_1,s_2;1,\Lambda ) \\
&=\frac {\zeta (s_1 +s_2 -1)}{s_2 -1}-(\log 2\pi ) \zeta (s_1+s_2)\nonumber\\
&+\sum ^{N-1}_{k=1 \atop k:{\rm odd}} \binom {-s_2}k M(-k)\zeta (s_1+s_2+k)\nonumber \\
&-\sum ^{N-1}_{k=1 \atop k:{\rm even}} \bigg[ \binom {-s_2}k\left\{ \left( -{a_k}+k!b_k\right)\zeta (s_1+s_2+k)-\zeta'(s_1+s_2+k)\right\} \nonumber \\
& \qquad\qquad -\frac {1}{k!}\frac{\Gamma'(s_2+k)}{\Gamma(s_2)}\zeta(s_1+s_2+k)\bigg]\nonumber\\
&-\frac{1}{\Gamma (s_2)}\sum _{n=1}^\infty \od(\rho _n)\Gamma (s_2-\rho _n)\Gamma (\rho _n)\zeta (s_1+s_2-\rho _n)\nonumber  \\
&+\frac{1}{2\pi i\Gamma (s_2)}\int_{(N-\varepsilon )}\Gamma (s_2+z)\Gamma (-z)M(-z)\zeta (s_1+s_2+z)dz,\nonumber 
\end{align}
where $N\in \mathbb N$, %satisfies $N>2, N>1-\Re s_2, N>2-\Re (s_1+s_2)$, 
$\varepsilon$ is a small positive number, and $\od(\rho _n)$ is the order of $\rho _n$ as a zero of the Riemann zeta function.
\end{thm}

\begin{proof}%[Proof of Theorem \ref{state1}] 
We first assume $\Re s_2>1,\Re (s_1+s_2)>2$ and ${\rm max}\{ -\Re s_2 ,1-\Re (s_1+ s_2 ) \} <c<-1$. Then we have 
\begin{eqnarray}\label{tr:phi(1,l)}
\Phi _2 (s_1,s_2;1,\Lambda )&=&\sum _{m=1}^\infty \sum _{n=1}^\infty \frac {\Lambda (n)}{m^{s_1} (m+n)^{s_2}} \\
&=&\sum _{m=1}^\infty \sum _{n=1}^\infty \frac {\Lambda (n)}{m^{s_1+s_2}} \left( 1+\frac nm\right) ^{-s_2}.\nonumber
\end{eqnarray}

We apply the Mellin-Barnes formula (\ref{MB}) with $s=s_2,\lambda ={n}/{m}$ to (\ref{tr:phi(1,l)}) to obtain
\begin{align}\label{continuation}
\Phi _2 (s_1,s_2;1,&\Lambda )=\sum _{m=1} ^\infty \sum _{n=1}^\infty \frac {\Lambda (n)}{m^{s_1+s_2}} \frac 1 {2\pi i}\int _{(c)} \frac{\Gamma (s_2 +z)\Gamma (-z)}{\Gamma (s_2)} \left( \frac nm\right) ^{z}dz\\
&=\frac 1 {2\pi i}\int _{(c)} \frac{\Gamma (s_2 +z)\Gamma (-z)}{\Gamma (s_2)} \sum _{n=1}^\infty \la (n)n^{z}\sum _{m=1} ^\infty m^{-s_1-s_2-z} dz\nonumber \\
&=\frac 1 {2\pi i}\int _{(c)} \frac{\Gamma (s_2 +z)\Gamma (-z)}{\Gamma (s_2)} M(-z)\zeta (s_1+s_2+z)dz. \nonumber 
\end{align}
The exchange between the summation and the integration is valid because of the choice of $c$.

Now, by the same argument as in \cite[Section 3]{MatsumotoEgami}, we shift the path $(c)$ to ($N-\varepsilon$) for an arbitrarily large $N>0$. In order to check the validity of this shifting, we prepare the known order estimations of $\Gamma (s)$, $\zeta (s)$ and $M(s)$ as follows.   First, we recall that
\begin{equation}\label{ST}
\Gamma (\sigma +it)=O(|t|^{ \sigma -\frac {1}{2}}e^{-\frac {\pi }{2}|t|}),
\end{equation}
\begin{equation}\label{order:zeta}
\zeta (\sigma +it)=O(|t|^A)\quad \ (A:{\rm positive\ constant})
\end{equation}
hold uniformly in fixed vertical strips (see \cite[$\S$ 12.3]{WW} and \cite[p.95 (5.1.1)]{Titch}). 
Next we can find an arbitrarily large $t_0>0$ such that
\begin{equation}
|t_0-\Im(\rho_n)|\gg (\log t_0)^{-1} \label{EM-3-5}
\end{equation}
for any $n\in \mathbb{N}$ (see \cite[(3.5)]{MatsumotoEgami}). Then it holds that
\begin{equation}
M(\sigma+it_0) \ll (\log t_0)^2\quad (-1\leq \sigma \leq 2) \label{EM-3-7}
\end{equation}
(see \cite[(3.7)]{MatsumotoEgami} or Ingham \cite[Theorem 26]{Ingham}). Furthermore it holds that
\begin{equation}
M(\sigma+it) \ll \log (2|\sigma+it|)\quad (\sigma<-1, t>0), \label{Ivic}
\end{equation}
provided that discs of radii $1/2$ around the trivial zeros $s=-2m$ $(m\in \mathbb{N})$ of $\zeta(s)$ are excluded (see Ivi\'c \cite[(12.22)]{Ivic} or \cite[Theorem 27]{Ingham}). Using these results, we see that the above shifting is possible by the same argument as in \cite[Section 3]{MatsumotoEgami}.

In the  course of this shifting, we encounter the poles of the integrand which are derived from $\Gamma (-z)$ and $M(-z)$. All poles of $\Gamma (-z)$ and $M(-z)$ are simple and the residues at those poles are 
\begin{center}
$\displaystyle{\Res_{z=-1}}\ M(-z)=-1$, 
\end{center}
\begin{center}
$\displaystyle{\Res_{z=-\rho _n}}M(-z)=\od(\rho _n)$,
\end{center}
\begin{center}
$\displaystyle{\Res_{z=k}}\ M(-z)=1\ \ (1\leq k\leq N-1,\ k:{\rm even})$
\end{center}and
\begin{equation*}
\displaystyle{\Res_{z=k}}\ \Gamma (-z)=\frac {(-1)^{k-1}}{k!} \ \ \ \ (0\leq k\leq N-1).
\end{equation*}
Therefore the integrand has 
simple poles at $z=-1, z=0, z=\rho_n$ ($n\in\mathbb{N}$), $z=k$
$(1\leq k\leq N-1,\ k:{\rm odd})$, and has
double poles at $z=k$ $(2\leq k\leq N-1,\ k:{\rm even})$ whose residues are
\begin{align*}
&\frac {1}{\Gamma (s_2)}\bigg\{ \Gamma (s_2+k)\left( \frac {(-1)^{k-1}}{k!} a_k +b_k\right) \zeta (s_1+s_2+k)\\
& \qquad\quad +\frac{(-1)^{k-1}}{k!}\left(\Gamma(s_2+k)\zeta'(s_1+s_2+k)+\Gamma'(s_2+k)\zeta(s_1+s_2+k)\right)\bigg\}\\
&=\binom {-s_2}k\left\{\left( -{a_k}+k!b_k\right)\zeta (s_1+s_2+k)-\zeta'(s_1+s_2+k)\right\}\\
& \qquad\quad  -\frac{1}{k!}\frac{\Gamma'(s_2+k)}{\Gamma(s_2)}\zeta(s_1+s_2+k),
\end{align*}
%where we note that $a_k$'s and $b_k$'s are the constant terms of the Laurent series expansion of $M(-z)$ and $\Gamma (-z)$ at $z=k$, respectively. 
hence we obtain \eqref{mainthm}. The first, the second, the third and the fourth terms on the right-hand side of (\ref{mainthm}) are %is 
obviously meromorphic on the whole space $\c ^2$.

The fifth term is convergent absolutely for all $s_1,s_2\in \c$ except for its singularities $s_1+s_2=1+\rho _n$ and $s_2=\rho _n-j$ ($j\in \n _0$) because of the order estimations (\ref{ST}), (\ref{order:zeta}) and
\begin{equation*}
\od(\rho _n)=O(\log \abs{\rho _n}),
\end{equation*}
which can be easily obtained from 
\begin{equation*}
N(T+1)-N(T)=O(\log T),
\end{equation*}
where $N(T)$ is the number of non-trivial zeros (counted with multiplicity) of $\zeta (\sigma +it)$ in the region $-T\leq t\leq T$ (see \cite[Chapter 9, p.214 (9.4.3)]{Titch}).

The integral on the right-hand side of (\ref{mainthm}) can be analytically continued to the region 
\begin{equation}
\mathcal{D}_N=\{ (s_1,s_2)\in \mathbb C ^2|\ \Re s_2>-N+\varepsilon,\Re (s_1+s_2)>1-N+\varepsilon \}, \label{Domain}
\end{equation}
because in this region the poles of the integrand are not on the path of integration. Since $N$ is arbitrary, (\ref{mainthm}) gives the analytic continuation of $\Phi _2 (s_1,s_2;1,\Lambda )$ to the whole space $\c ^2$. Thus we complete the proof of Theorem \ref{state1}.
\end{proof}

From Theorem \ref{state1}, we can prove the following.

\begin{thm}\label{C-2-1}\ 
The singularities of $\Phi _2 (s_1,s_2;1,\Lambda )$ are located 
only on the subsets of $\c^2$ defined by one of the following equations:
\begin{align}
& s_2=1,  \label{sing-1}\\
& s_2=-l\quad (l\in \mathbb{N},\ l\geq 2),\label{sing-2}\\
& s_1+s_2=2-l\quad (l\in \mathbb{N}_0),\label{sing-3}\\
& s_2=-l+\rho_n\quad (n\in \mathbb{N},\ l\in \mathbb{N}_0),\label{sing-4}\\
&s_1+ s_2=1+\rho_n\quad (n\in \mathbb{N}),\label{sing-5}
\end{align}
all of which are ``true'' singularities. 
\end{thm}

\begin{proof}
From \eqref{mainthm} it is easy to see that possible singularities of
$\Phi _2 (s_1,s_2;1,\Lambda )$ are located only on the hyperplanes defined by the
above list.

We prove that all of those are true singularities. 
Fix an arbitrary $N\in \mathbb{N}$ and consider \eqref{mainthm} on $\mathcal{D}_N$ defined by \eqref{Domain}. 
We put $s_1+s_2 = u$, and regard \eqref{mainthm} as a formula in variables $u,s_2$. 
This idea of ``changing variables'' is originally due to Akiyama, Egami and Tanigawa \cite{AET}. For simplicity, we write the right-hand side of \eqref{mainthm} as 
$$X_1-X_2+\sum_{k=1 \atop k:\text{odd}}^{N-1}Y_{1,k} -\sum_{k=1 \atop k:\text{even}}^{N-1}Y_{2,k} -\sum_{n=1}^\infty Z_n +\text{(integral part)}.$$
We can see that the integral part is absolutely convergent for any $u,s_2\in \mathcal{D}_N$, namely is holomorphic. 
We have possible singularities $u=2$ and $s_2=1$ arising from $X_1$, $u=1$ from $X_2$, $u=1-k$ ($k$:\,odd) from $Y_{1,k}$.
 
Next consider $Y_{2,k}$.    The possible singularities coming from this term are
$u=1-k$ ($k$:\,even) and $s_2=-k-m$ ($k$:\,even, $m\in\mathbb{N}_0$).
The Laurent series of $Y_{2,k}$ around $u=1-k$ has the principal part of order 2 
because of the term $\zeta'(u+k)$, hence this is indeed a candidate of singularity.
Also, since
\begin{equation}
\frac{\Gamma'(s+k)}{\Gamma(s)}=\frac{(-1)^{k-1}(k+m)!}{m!}\frac{1}{s+k+m}+O(1)\quad (s \to -k-m) \label{Gamma-property}
\end{equation}
for even $k\geq 2$ and $m\in \mathbb{N}_0$, the principal part of the Laurent series of $Y_{2,k}$ around $s_2=-l$ $(l\in \mathbb{N},\ l\geq 2)$ is 
$$\sum_{k+m=l \atop \text{$k$:\,even}}\frac{\zeta(u+k)}{k!}
\frac{(-1)^{k-1}l!}{m!}\frac{1}{s_2+l}=
\sum_{k=2 \atop \text{$k$:\,even}}^{l}\,(-1)^{k-1}\binom{l}{k}\zeta(u+k)\frac{1}{s_2+l},$$
whose coefficient does not vanish identically as a function in $u$.   In fact, when $l$ is even, by setting $u=-l$, the coefficient equals to 
$$\sum_{k=2 \atop \text{$k$:\,even}}^{l}\,(-1)^{k-1}\binom{l}{k}
\zeta(-l+k)=(-1)^{l-1}\zeta(0) \neq 0,$$
because negative even integers are trivial zeros of the Riemann zeta-function.
When $l$ is odd, by setting $u=-l+1$, the coefficient equals to 
$$(-1)^l\binom{l}{l-1}\zeta(0) \neq 0.$$
Therefore $s_2=-l$ $(l\in \mathbb{N},\ l\geq 2)$ are candidates of singularities. 
From $Z_n$, we have possible singularities $s_2=-l+\rho_n$ $(l\in \mathbb{N},\ l\geq 2)$. Consequently we obtain the list \eqref{sing-1}-\eqref{sing-5},
and all of which cannot be cancelled 
each other. Therefore all of them are true singularities. 
\end{proof}

\begin{remark}\label{Rem-2}\ 
Akiyama, Egami and Tanigawa observed that sets of singularities of $\zeta_r(s_1,\ldots,s_r)$ include many points of indeterminacy, like $(s_1,s_2)=(0,0)$ of the function $s_1/(s_1+s_2)$ (see \cite[Section 3]{AET}). At those points, the values of $\zeta_r(s_1,\ldots,s_r)$ can be determined only as a limit value, depending on a choice of limiting process. 
As for $\Phi _2 (s_1,s_2;1,\Lambda )$, we can see that the similar situation occurs from \eqref{mainthm}. In later sections, we will observe this type of special values.
\end{remark}

%%%%%%%%%%%%%%%%%%%%%%%%%%%%%%%%%%%%%%%%%%%%%%%%%%%%%%%%%%%%%
\section{The double series $ \Phi _2 (s_1,s_2;1,\widetilde{\alpha} )$}\label{sec-3}
%%%%%%%%%%%%%%%%%%%%%%%%%%%%%%%%%%%%%%%%%%%%%%%%%%%%%%%%%%%%%

In this section, by the same principle as in the previous section, we consider a general class of double series $\Phi _2 (s_1,s_2;1,\widetilde{\alpha} )$ defined by 
\eqref{def_double_tilde}.
It is well-known that
\begin{equation}
\frac{1}{\zeta(s)}=\sum _{n=1}^\infty \frac {\mu(n)}{n^s}\quad (\Re s>1),\label{mu-Dir}
\end{equation}
where $\mu$ is 
the M{\"o}bius function defined for $n\in \n$ by
\begin{eqnarray*}
\mu (n)=\left\{ \begin{array}{ll}
1 & (n=1),\\
(-1)^r & (n :{\rm squarefree\ and}\ n=p_1\cdots p_r,\ p_i\in \p ),\\
0 & (n :{\rm not\ squarefree}) \\
\end{array} \right. 
\end{eqnarray*}
(see \cite[$\S$ 1.1]{Titch}). Hence, for $\Re s>\max\{1,\delta\}$, 
we have
$$\frac{\Phi (s;\alpha)}{\zeta(s)}=\sum _{m=1}^\infty \sum _{n=1}^\infty\frac {{\alpha}(m)\mu(n)}{(mn)^s}.$$
Therefore \eqref{def_tilde_alpha} implies 
\begin{equation}
\widetilde{\alpha}(n)=\sum_{1\leq d \leq n \atop d\mid n}\alpha\left(\frac{n}{d}\right)\mu(d). \label{Rev-form}
\end{equation}
%Now we consider 
%\begin{equation}\label{def:phi-tilde}
%\Phi _2 (s_1,s_2;1,\widetilde{\alpha} )=\sum _{m_1=1}^\infty \sum _{m_2=1}^\infty \frac {\widetilde{\alpha} (m_2)}{m_1^{s_1} (m_1+m_2)^{s_2}}.
%\end{equation}

%The following result is a generalization of the known fact about the region of absolute convergence of the double zeta function of Euler-Zagier type (see \cite[(3.2)]{Matsumoto}). 

\begin{prop}\label{Prop-region}
If $\alpha$ satisfies the condition (I) of the class $\mathcal{A}$, we see that 
$\Phi _2 (s_1,s_2;1,\widetilde{\alpha} )$ is absolutely convergent in the region
\begin{equation}
\left\{(s_1,s_2)\in \mathbb{C}^2 \mid \Re s_2>\max\{1,\delta\},\ \Re (s_1+s_2)>\max\{ 2,1+\delta\}\right\}.\label{conv-region}
\end{equation}
\end{prop}

\begin{proof}\ Set $\sigma_j=\Re s_j$ for $j=1,2$. By \eqref{Rev-form}, we have
$$|\widetilde{\alpha}(n)|\leq \sum_{1\leq d \leq n \atop d\mid n}|\alpha(d)|.$$
Hence we have
\begin{align*}
& \sum_{m=1}^\infty \sum_{n=1}^\infty \left| \frac{\widetilde{\alpha}(n)}{m^{s_1}(m+n)^{s_2}}\right|\leq \sum_{m=1}^\infty \sum_{n=1}^\infty  \frac{\sum_{d \mid n}|\alpha(d)|}{m^{\sigma_1}(m+n)^{\sigma_2}}\\
& \quad =\sum_{d=1}^\infty |\alpha(d)|\sum_{m=1}^\infty \sum_{l=1}^\infty  \frac{1}{m^{\sigma_1}(m+dl)^{\sigma_2}}\\
& \quad \leq \sum_{d=1}^\infty |\alpha(d)|\left(\sum_{m,l\geq 1 \atop m\leq dl}\frac{1}{m^{\sigma_1}(dl)^{\sigma_2}}+\sum_{m,l\geq 1 \atop m> dl}\frac{1}{m^{\sigma_1+\sigma_2}}\right)\\
& \quad = \sum_{d=1}^\infty \frac{|\alpha(d)|}{d^{\sigma_2}}\sum_{l=1}^\infty \frac{1}{l^{\sigma_2}}\sum_{1\leq  m\leq dl}\frac{1}{m^{\sigma_1}}+\sum_{d=1}^\infty |\alpha(d)|\sum_{l=1}^\infty \sum_{m> dl}\frac{1}{m^{\sigma_1+\sigma_2}}\\
& \quad =:\Sigma_1+\Sigma_2,
\end{align*}
say.
As for $\Sigma_2$, we first assume $\sigma_1+\sigma_2>1$. Then we have
\begin{align*}
& \sum_{m> dl}\frac{1}{m^{\sigma_1+\sigma_2}} \leq  \int_{dl}^\infty \frac{dt}{t^{\sigma_1+\sigma_2}}=\left[\frac{t^{1-\sigma_1-\sigma_2}}{1-\sigma_1-\sigma_2}\right]_{dl}^{\infty} \ll (dl)^{1-\sigma_1-\sigma_2}. 
\end{align*}
Hence we obtain
\begin{align*}
\Sigma_2& \ll \sum_{d=1}^\infty |\alpha(d)|\sum_{l=1}^\infty (dl)^{1-\sigma_1-\sigma_2}\\
& = \sum_{d=1}^\infty \frac{|\alpha(d)|}{d^{\sigma_1+\sigma_2-1}}\sum_{l=1}^\infty \frac{1}{l^{\sigma_1+\sigma_2-1}}.
\end{align*}
By the condition (I) of the class $\mathcal{A}$, 
%Assumption \ref{Ass-1} (I), 
we see that the first and the second sums are convergent for $\sigma_1+\sigma_2-1>\delta$ and $\sigma_1+\sigma_2-1>1$, respectively. Therefore $\Sigma_2$ is convergent for 
\begin{equation}
\sigma_1+\sigma_2>\max \{2,1+\delta\}.  \label{S2-region}
\end{equation}
As for $\Sigma_1$, we remark that
\begin{equation*}
\sum_{1\leq  m\leq dl}\frac{1}{m^{\sigma_1}}\ll 
\begin{cases}
1 & (\sigma_1>1),\\
\log (dl) & (\sigma_1=1),\\
(dl)^{1-\sigma_1} & (\sigma_1<1).
\end{cases}
\end{equation*}
If $\sigma_1>1$, we have
$$\Sigma_1 \ll \sum_{d=1}^\infty \frac{|\alpha(d)|}{d^{\sigma_2}}\sum_{l=1}^\infty \frac{1}{l^{\sigma_2}},$$
where the first and the second sums are convergent for $\sigma_2>\delta$ and $\sigma_2>1$, respectively. Hence $\Sigma_1$ in this case is convergent for 
$\sigma_2>\max \{1,\delta\}$. 

If $\sigma_1=1$, we have
\begin{align*}
\Sigma_1 & \ll \sum_{d=1}^\infty \frac{|\alpha(d)|}{d^{\sigma_2}}\sum_{l=1}^\infty \frac{\log (dl)}{l^{\sigma_2}}\\
& \ll \sum_{d=1}^\infty \frac{|\alpha(d)|\log d}{d^{\sigma_2}}\sum_{l=1}^\infty \frac{1}{l^{\sigma_2}}+ \sum_{d=1}^\infty \frac{|\alpha(d)|}{d^{\sigma_2}}\sum_{l=1}^\infty \frac{\log l}{l^{\sigma_2}}.
\end{align*}
As well as the case $\sigma_1>1$, we see that $\Sigma_1$ in this case is convergent for 
$\sigma_2>\max \{1,\delta\}$. 

If $\sigma_1<1$, we have
\begin{align*}
\Sigma_1 & \ll \sum_{d=1}^\infty \frac{|\alpha(d)|}{d^{\sigma_2}}\sum_{l=1}^\infty \frac{(dl)^{1-\sigma_1}}{l^{\sigma_2}}= \sum_{d=1}^\infty \frac{|\alpha(d)|}{d^{\sigma_1+\sigma_2-1}}\sum_{l=1}^\infty \frac{1}{l^{\sigma_1+\sigma_2-1}}.
\end{align*}
Hence we can similarly see that $\Sigma_1$ in this case is convergent for 
$\sigma_1+\sigma_2>\max \{2,1+\delta\}$. 

Combining these three cases, we see that $\Sigma_1$ is convergent for 
\begin{equation}
\sigma_1\geq 1,\ \sigma_2>\max \{1,\delta\}  \label{S1-region-1}
\end{equation}
or 
\begin{equation}
\sigma_1< 1,\ \sigma_1+\sigma_2>\max \{2,1+\delta\}.  \label{S1-region-2}
\end{equation}
We obviously see that the region \eqref{S2-region} includes both the regions \eqref{S1-region-1} and \eqref{S1-region-2}. Hence, denoting by $\Omega$ the union of regions \eqref{S1-region-1} and \eqref{S1-region-2}, we see that $\Phi _2 (s_1,s_2;1,\widetilde{\alpha} )$ is absolutely convergent in $\Omega$. We remark that 
\begin{equation*}
\Omega=
\begin{cases}
\{ \Re s_1\geq 1,\ \Re s_2>1\} \cup \{ \Re s_1< 1,\ \Re s_1+\Re s_2>2\} & (\delta<1),\\
\{ \Re s_1\geq 1,\ \Re s_2>\delta\} \cup \{ \Re s_1< 1,\ \Re s_1+\Re s_2>1+\delta\} & (\delta\geq 1).
\end{cases}
\end{equation*}
Therefore we see that $\Omega$ is equal to the region \eqref{conv-region}. This completes the proof.
\end{proof}

\begin{remark}
It is known that $\zeta_2(s_1,s_2)$ is absolutely convergent in the region
$\Re s_2>1, \Re(s_1+s_2)>2$ (see \cite[(3.2)]{Matsumoto}).
Proposition \ref{Prop-region} gives a generalization of this fact.
\end{remark}

%%%%%%%%%%%%%%%%%%%%%%%%%%%%%%%%%%%%%%%%%%%%%%%%%%%%%%%%%%%%%
\section{Analytic continuation of $\Phi _2 (s_1,s_2;1,\widetilde{\alpha} )$}\label{sec-4}
%%%%%%%%%%%%%%%%%%%%%%%%%%%%%%%%%%%%%%%%%%%%%%%%%%%%%%%%%%%%%

In this section, we assume $\alpha\in\mathcal{A}$, 
Moreover we assume the following condition for $\zeta(s)$.

\begin{assumption}\label{Ass-2}\ Assume that all non-trivial zeros of $\zeta(s)$ are simple and that
\begin{equation}\label{order:1/zeta'(rho)}
\frac 1{\zeta '(\rho _n)}=O(\abs {\rho _n}^B)\quad (n\to \infty),
\end{equation}
with some constant $B> 0$. 
%Consequently all $\rho_n$'s are simple zeros.
\end{assumption}

Note that (\ref{order:1/zeta'(rho)}) can be regarded as a quantitative version of the
well-known ``simplicity conjecture'' on the zeros of $\zeta(s)$, and also, 
as a weak version of 
\begin{equation*}
\sum _{0<\Im {\rho _n}\leq T}\frac 1{\abs {\zeta '(\rho _n)}^{2k}}=O(T(\log T)^{(k-1)^2})\ \ (k\in \r;\,k>0 )
\end{equation*}
with $k=1/2$, which is conjectured independently by Gonek \cite{Gonek} and Hejhal \cite{Hejhal}. 

Under the above assumptions, we
give the analytic continuation of 
$\Phi _2(s_1,s_2;1,\tilde{\alpha} )$.
For this aim, we first recall some properties of $1/\zeta (s)$.

Using the functional equation of $\zeta (s)$ (see \cite[$\S$ 2.1]{Titch}):
\begin{equation}\label{feq:zeta}
\zeta (s)=2\Gamma (1-s)\sin \frac {\pi s}2 (2\pi )^{s-1}\zeta (1-s),
\end{equation}
we can obtain
\begin{align}\label{value:zeta'(-2k)}
\zeta '(-2k)&=\frac \pi 2\cos \left( \frac {-2k\pi} 2\right) 2\Gamma (1+2k)(2\pi )^{-2k-1}\zeta (1+2k)\\
&=\frac{(-1)^k (2k)!}{2(2\pi )^{2k}}\zeta (1+2k)\neq 0\ \ (k\in \n ).\nonumber
\end{align}
Therefore all trivial zeros of $\zeta (s)$ produce simple poles of $1/\zeta (s)$. Hence we obtain the following. 

\begin{lem}\label{residue:1/zeta}
Under Assumption \ref{Ass-2}, 
all poles of $1/\zeta (s)$ are simple and the residues are
\begin{equation*}
\displaystyle{\Res_{s=\rho _n}}\frac 1{\zeta (s)}=\frac 1{\zeta '(\rho _n)},
\end{equation*}
\begin{equation*}
\displaystyle{\Res_{s=-2k}}\ \frac 1{\zeta (s)}=\frac 1{\zeta '(-2k)}=\frac {(-1)^k 2(2\pi )^{2k}}{(2k)!\zeta (1+2k)}\ \ (k\in \n ).
\end{equation*}
\end{lem}

\begin{proof}
By the above fact and Assumption \ref{Ass-2}, 
the simplicity of each zero is obvious. Therefore we have
\begin{equation*}
\displaystyle{\Res_{s=\rho _n}}\frac 1{\zeta (s)}=\lim _{s\to \rho _n}\frac {s-\rho _n}{\zeta (s)}=\frac 1{\zeta '(\rho _n)},
\end{equation*}
\begin{equation*}
\displaystyle{\Res_{s=-2k}}\ \frac 1{\zeta (s)}=\lim _{s\to -2k}\frac {s+2k}{\zeta (s)}=\frac 1{\zeta '(-2k)}.
\end{equation*}
Substituting (\ref{value:zeta'(-2k)}) into the second equation, we complete the proof of Lemma \ref{residue:1/zeta}.
\end{proof}

Applying the Mellin-Barnes formula (\ref{MB}), we obtain 
\begin{align*}
\Phi _2 (s_1,s_2;1,\widetilde{\alpha} )
& \ =\sum _{m,n=1}^\infty \frac {\widetilde{\alpha} (n)}{m^{s_1+s_2}}\frac{1}{2\pi i\Gamma(s_2)}\int_{(c)}\Gamma(s_2+z)\Gamma(-z)\left(\frac{n}{m}\right)^zdz\\
& \ =\frac{1}{2\pi i\Gamma(s_2)}\int_{(c)}\Gamma(s_2+z)\Gamma(-z)\Phi(-z;\widetilde{\alpha})\zeta(s_1+s_2+z)dz,
\end{align*}
where $\Re s_2>\max\{1,\delta\},\Re (s_1+s_2)>\max\{2,1+\delta\}$ and 
$\max\{-\Re(s_2),\,1-\Re(s_1+s_2)\} < c < \min\{-1,-\delta\}<0$. 

As in the proof of Theorem \ref{state1}, we shift the path $(c)$ to ($N-\varepsilon$), and apply 
Lemma \ref{residue:1/zeta} to $\Phi _2(s_1,s_2;1,\widetilde{\alpha} )$. 
Here we recall the definition of the Bernoulli numbers $\{B_n\}_{n\geq 0}$, which are given by
\begin{equation*}
\frac {te^t}{e^t-1}=\sum _{n=0}^\infty B_n\frac {t^n}{n!}.
\end{equation*}
Then we can obtain
\begin{align}\label{mainthm2}
&\Phi _2(s_1,s_2;1,\widetilde{\alpha} )\\
&=\frac{\Gamma(s_2-\delta)\Gamma(\delta)}{\Gamma(s_2)} \Res_{s=\delta}\ \left(\frac{\Phi(s;\alpha)}{\zeta(s)}\right)\zeta(s_1+s_2-\delta) \notag\\
& \ -2\Phi(0;\alpha)\zeta (s_1+s_2)-\sum _{k=1 \atop k:{\rm odd}}^{N-1}\binom {-s_2}k \frac {(k+1)\Phi(-k;\alpha)}{B_{k+1}}\zeta (s_1+s_2+k)\nonumber \\
&\ +\sum _{k=1\atop k:{\rm even}}^{N-1}  \bigg[ \binom {-s_2}k\bigg\{ \frac {(-1)^{k/2}2(2\pi )^{k}\Phi(-k;\alpha)}{\zeta (1+k)}\notag \\
& \ \ \times \left(b_k \zeta(s_1+s_2+k)-\frac{1}{k!}\zeta'(s_1+s_2+k)\right)+c_k(\alpha)\zeta(s_1+s_2+k)\bigg\}\nonumber \\
& \qquad -\frac{\Gamma'(s_2+k)}{\Gamma(s_2)}\frac {(-1)^{k/2}2(2\pi )^{k}\Phi(-k;\alpha)}{(k!)^2\zeta (1+k)}\zeta(s_1+s_2+k)\bigg]\notag\\
&\ +\frac 1 {\Gamma (s_2)}\sum _{n=1}^\infty \Gamma (s_2-\rho _n)\Gamma (\rho _n)\frac {\Phi(\rho_n;\alpha)}{\zeta'(\rho _n)}\zeta (s_1+s_2-\rho _n)\nonumber \\
&\ +\frac 1 {2\pi i\Gamma (s_2)}\int _{(N-\varepsilon )}\Gamma (s_2+z)\Gamma (-z)\Phi(-z;\widetilde{\alpha})\zeta (s_1+s_2+z)dz\nonumber
\end{align}
for $N\in \mathbb N$, where $\varepsilon$ is a small positive number, $b_k$ is defined by \eqref{def-bk} and $c_k(\alpha)$ is the constant term of the Laurent series expansion of $\Phi(s;\alpha)/\zeta (s)$ at $s=-k$ for any even positive integer $k$, namely
\begin{equation}
 c_k(\alpha)=\frac 1{2\pi i}\int _{|\xi +k|=\frac 1 2} \frac {\Phi(s;\alpha)}{\zeta(\xi)}\frac{1}{\xi +k} d\xi\quad (k\in \n;\ k:\text{even}). \label{def-ck}
\end{equation}
The validity of shifting the path of the integration can be shown as in Section \ref{sec-2}.
In fact, we can find an arbitrarily large $t_1>0$ for which
\begin{equation}\label{order:1/zeta}
\frac 1{\zeta (\sigma +it_1)}=O(|t_1|^A)\ \ (A:{\rm a\ positive\ constant})
\end{equation}
holds for $-1\leq\sigma\leq 2$ (see \cite[p.218, Theorem 9.7]{Titch}), while we can
easily see by the functional equation that $1/\zeta(s)$ is of polynomial order in $|t|$
in the half-plane $\sigma\leq -1$ as $|t| \to \infty$.    Using these facts and the condition (III) of the
class $\mathcal{A}$, we find that the argument in Section \ref{sec-2} works again.
 
The second, the third and the fourth term on the right-hand side of (\ref{mainthm2}) are meromorphic on the whole space $\c ^2$. The last integral can be analytically continued to the region $\{ (s_1,s_2)\in \mathbb C ^2|\ \Re s_2>-N+\varepsilon,\Re (s_1+s_2)>1-N+\varepsilon \}$ by using the condition (III)
 and noting that in this region the poles of the integrand are not on the path of integration. 
It follows from (\ref{ST}), (\ref{order:zeta}) and Assumptions \ref{Ass-2} that the fifth term on the right-hand side of (\ref{mainthm2}) is convergent absolutely for all $(s_1,s_2)\in \c ^2$ except for its singularities $s_2=-j+\rho _n$ ($j\in \n _0$) and $s_1+s_2=1+\rho _n$. Finally we note that the first term is also meromorphic on $\c ^2$ and particularly vanishes when $\Phi(s;\alpha)$ has no pole or $\delta=1$ except for the case $s_2=1$. 

The above arguments lead to the following theorem. The last assertion can be proved by the same way as in Theorem \ref{C-2-1}.

\begin{thm}\label{Th-4-2}
Let $\alpha\in\mathcal{A}$.
Under Assumption \ref{Ass-2}, $\Phi _2(s_1,s_2;1,\widetilde{\alpha} )$ can be continued meromorphically to the whole space $\c ^2$ by {\rm (\ref{mainthm2})}.  
The possible singularities of $\Phi _2 (s_1,s_2;1,\widetilde{\alpha} )$ are located 
only on the subsets of $\c^2$ defined by one of the following equations:
\begin{equation}
\begin{split}
& s_1+s_2=1-k\quad (k\in \mathbb{N}_0),\\
& s_2=-k\quad (k\in \mathbb{N},\ k\geq 2),\\
& s_2=-l+\rho_n\quad (l\in \mathbb{N}_0,\ n\in \mathbb{N}),\\
& s_1+ s_2=1+\rho_n\quad (n\in \mathbb{N}),\\
& s_2=-l+\delta\quad  (l\in \mathbb{N}_0),\\
& s_1+ s_2=1+\delta,
\end{split}
\label{coro-2}
\end{equation}
where the last two equations are omitted when $\Phi(s;\alpha)$ has no pole or $\delta=1$. All of them are ``true'' singularities.
\end{thm}

%%%%%%%%%%%%%%%%%%%%%%%%%%%%%%%%%%%%%%%%%%%%%%%%%%%%%%%%%%%
\section{Reverse values of $\Phi _2 (s_1,s_2;1,\Lambda )$}\label{sec-5}
%%%%%%%%%%%%%%%%%%%%%%%%%%%%%%%%%%%%%%%%%%%%%%%%%%%%%%%%%%

As stated in Remark \ref{Rem-2}, Akiyama, Egami and Tanigawa \cite{AET} observed that the value of multiple zeta function \eqref{EZ} at any point $(-m_1,\ldots,-m_r)$ $(m_1,\ldots,m_r\in \n_0)$ on the sets of singularities is what is called a point of indeterminacy, whose ``value'' can be understood only as a limit value
which depends on a limiting process. They considered %$\zeta_r(-m_1,\ldots,-m_r)$ by 
\begin{align*}
& \lim_{s_1\to -m_1}\lim_{s_2\to -m_2}\cdots \lim_{s_r\to -m_r}\zeta_r(s_1,s_2,\ldots,s_r),
\end{align*}
which they called the \textit{regular value}.   Akiyama and Tanigawa \cite{AT} further defined two different types of limit values at $(-m_1,\ldots,-m_r)$ by
\begin{align*}
& \lim_{s_r\to -m_r}\cdots \lim_{s_2\to -m_2} \lim_{s_1\to -m_1}\zeta_r(s_1,s_2,\ldots,s_r),\\
& \lim_{\varepsilon \to 0}\ \zeta_r(-m_1+\varepsilon,-m_2+\varepsilon,\ldots,-m_r+\varepsilon)
\end{align*}
which are called the \textit{reverse value} and the \textit{central value}, respectively 
(for generalizations, see also Sasaki \cite{Sasaki} \cite{Sasaki2}, Komori \cite{Komori}, 
Onozuka \cite{Onozuka}, and Onozuka, Wakabayashi and the first-named author \cite{MOW}). 
%In fact, each $(-m_1,\ldots,-m_r)$ is a point of indeterminacy for $\zeta_r(s_1,\ldots,s_r)$, like $(s_1,s_2)=(0,0)$ of the function $s_1/(s_1+s_2)$ (see \cite[Section 3]{AET}). At this point, the value of the function depends on a limitting process. 

Checking the right-hand sides of \eqref{mainthm} and \eqref{mainthm2}, we can see that the same situation as above occurs for $\Phi _2 (s_1,s_2;1,\Lambda )$ and for $\Phi_2(s_1,s_2;1,\widetilde{\alpha})$. 
We aim to consider reverse values of them in this and the next section, because
reverse values seem to be more interesting than regular values in the present
situation (see Remark \ref{Rem-5-5}). 

Here we define the reverse value of $\Phi _2 (s_1,s_2;1,\Lambda )$ at $(u_1,u_2)$ on each singular set determined by %\eqref{coro-1}: 
\begin{align*}
& \Phi _2^{\rm Rev} (u_1,u_2;1,\Lambda )=\lim_{s_2\to u_2}\lim_{s_1\to u_1}\Phi _2 (s_1,s_2;1,\Lambda ).
\end{align*}
We give several examples of explicit formulas for reverse values by use of $\lim_{s\to -k}1/\Gamma(s)=0$ for $k\in \mathbb{N}_0$. 
Let $(s)_k:=s(s+1)(s+2)\cdots (s+k-1)$.

\begin{prop}\label{spv1} Let $m,n\in \n _0$ with $2\mid (m+n)$ and assume $m\geq 1$ when $n\geq 2$. Then 
\begin{align*}
&\Phi _2^{\rm Rev} (-m,-n;1,\Lambda )\\
&\ =\frac {B_{m+n+2}}{(n+1)(m+n+2)}+(\log 2\pi ) \frac{B_{m+n+1}}{m+n+1}\\
&\ \ -\sum _{k=1 \atop k:{\rm odd}}^n \binom n k \frac {k+1}{B_{k+1}}\zeta '(-k)\frac {B_{m+n-k+1}}{m+n-k+1}\\
&\ \ +\sum _{k=1\atop k:{\rm even}}^n  \binom n k \frac {B_{m+n-k+1}}{m+n-k+1}\left( -a_k+k!b_k\right) \\
& \qquad\quad -\frac{(-1)^m m!n!}{(m+n+1)!} M(-m-n-1).
\end{align*}
\if0
where
\begin{equation*}
R_{m,n}^1=\left\{ \begin{array}{ll}
0 & (m+n:{\rm odd}) \\
{(-1)^{m} m!n!}/{(m+n+1)!} & (m+n:{\rm even}) \\
\end{array} \right. 
\end{equation*}
\begin{equation*}
R_{m,n}^2=\left\{ \begin{array}{ll}
{(-1)^{m} m!n!}/{(m+n+1)!} & (m+n:{\rm odd}) \\
0 & (m+n:{\rm even}) \\
\end{array} \right. .
\end{equation*}
\fi
\end{prop}

\begin{proof} We write (\ref{mainthm}) as
\begin{align}\label{brief}
& \Phi _2(s_1,s_2;1,\Lambda )\\
%&=\frac {\zeta (s_1 +s_2 -1)}{s_2 -1}-(\log 2\pi ) \zeta (s_1+s_2) \\ 
%&+\sum ^{N-1}_{k=1 \atop k:{\rm odd}} \binom {-s_2}k M(-k)\zeta (s_1+s_2+k) \nonumber \\
%&-\sum ^{N-1}_{k=1 \atop k:{\rm even}} \binom {-s_2}k\zeta (s_1+s_2+k)\left( -{a_k}+k!b_k\right) \nonumber \\
%&-\frac{1}{\Gamma (s_2)}\sum _{n=1}^\infty \od(\rho _n)\Gamma (s_2-\rho _n)\Gamma (\rho _n)\zeta (s_1+s_2-\rho _n) \nonumber \\
%&+\frac{1}{2\pi i\Gamma (s_2)}\int_{(N-\varepsilon )}\Gamma (s_2+z)\Gamma (-z)M(-z)\zeta (s_1+s_2+z)dz,\nonumber \\
&\ =\frac {\zeta (s_1 +s_2 -1)}{s_2 -1}-(\log 2\pi ) \zeta (s_1+s_2) 
+\sum _o -\sum _e -\sum _{\rho} +I,\notag %({\rm say}). 
\end{align}
say.    We let $\lim_{s_2\to -n}\lim_{s_1\to -m}$ on the both sides of (\ref{brief}) under $N>m+n+1$. Then we see that $\sum _{\rho}\to 0$ and $I \to 0$ because $\lim_{s_2\to -n}|\Gamma(s_2)|\to \infty$.    
Also, applying the well-known fact 
\begin{equation*}
\zeta (1-j)=-\frac {B_{j}} {j}\ \ ( j\in \n )
\end{equation*}
(see \cite[$\S$ 2.4]{Titch}),  we obtain
\begin{align*}
\sum _o\to &-\sum _{k=1\atop k:{\rm odd}}^n \binom n k \frac {k+1}{B_{k+1}}\zeta '(-k)\frac {B_{m+n-k+1}}{m+n-k+1}\\
& \ - \frac{(-1)^m m! n!}{(m+n+1)!} M(-m-n-1)
\end{align*}
for $n\geq 1$, and
\begin{align*}
-\sum _e&\to \sum _{k=1\atop k:{\rm even}}^n \bigg[ \binom n k \left\{\frac {B_{m+n-k+1}}{m+n-k+1}\left( -a_k+k!b_k\right)-\zeta'(-m-n+k)\right\} \\
& \qquad\quad +\binom{n}{k}\zeta'(-m-n+k)\bigg]
\end{align*}
for $n\geq 2$ and $m\geq 1$. 
%Here, the terms including $R_{m,n}^1$ and $R_{m,n}^2$ come from the terms of $k=m+n+1$. 
%In fact, 
Here, in the calculations of $\sum_o$, 
the pole of $\zeta (s_1+s_2+k)$ for $k=m+n+1$ and $s_1=-m$ at $s_2=-n$ is cancelled with the zero of the binomial coefficient, namely,
\begin{align*}
&\binom {-s_2}{m+n+1} \zeta (s_2+n+1)\\
&=\frac {-s_2(-s_2-1)\cdots (-s_2-n+1)(-s_2-n)(-s_2-n-1)\cdots (-s_2-m-n)}{(m+n+1)!}\\
&\times \left( \frac 1 {s_2+n} +\cdots  \right),
\end{align*}
and so, letting $s_2\to -n$, we have
\begin{align*}
\binom {-s_2}{m+n+1} \zeta (s_2+n+1)
\to &-\frac {n(n-1)\cdots 1\cdot (-1)(-2)\cdots (-m)}{(m+n+1)!}\\
&\ = \frac {(-1)^{m-1}m!n!}{(m+n+1)!},%=-R_{m,n}^1\ {\rm or}\  -R_{m,n}^2.
\end{align*}
because $m+n$ is even. 
Also in the calculations of $\sum_e$, we use \eqref{Gamma-property}, 
we have
$$\lim_{s_2\to -n}-\frac {1}{k!}\frac{\Gamma'(s_2+k)}{\Gamma(s_2)}\zeta(-m+s_2+k)=\binom{n}{k}\zeta'(-m-n+k)$$
for $n\geq 2$ and $m\geq 1$, since $\zeta(-m-n+k)=0$. 
Thus we complete the proof of Proposition \ref{spv1}.
\end{proof}

\begin{remark}\label{R-5-2}
As stated above, Proposition \ref{spv1} holds for $m,n\in \mathbb{N}_0$ with $2\mid m+n$. If $m+n$ is odd with $n\geq 2$, then 
$\Phi_2(s_1,s_2;1,\Lambda)$ is not convergent as ${s_1\to -m}$ and ${s_2\to -n}$. In fact, in this case, the right-hand side of (\ref{mainthm}) is not convergent because of the term $\Gamma'(s_2+k)/\Gamma(s_2)$ for any even $k\ (\leq n)$. We emphasize that for the case $m+n$ is even, the pole of $\Gamma'(s_2+k)/\Gamma(s_2)$ at $s_2=-n$ and the zero of $\zeta(s_2-m+k)$ at $s_2=-n$ are cancelled and this determines finite values. 
\end{remark}

\begin{example}\label{Example-4-1}\ From Proposition \ref{spv1}, we obtain 
\begin{align*}
&\Phi _2^{\rm Rev} (0,0;1,\Lambda )=\frac 1 2\log 2\pi +\frac 1 {12} -12\zeta '(-1),\\
&\Phi _2^{\rm Rev} (-1,0;1,\Lambda )=\frac 1 {12}\log 2\pi -\frac{3}{4}-\frac{\zeta''(-2)}{4\zeta'(-2)}+\frac{\gamma}{2},\\
&\Phi _2^{\rm Rev} (-1,-1;1,\Lambda )=-\frac 1{240}-\zeta '(-1)-20\zeta '(-3).
%&\Phi _2^{\rm Rev} (0,0;1,\Lambda )=\frac 1 2\log 2\pi +\frac 1 {12} -12\zeta '(-1),\\
%&\Phi _2^{\rm Rev} (-1,0;1,\Lambda )=\frac 1 {12}\log 2\pi -\frac 1 2 (-a_2+2b_2),\\
%&\Phi _2^{\rm Rev} (-1,-1;1,\Lambda )=-\frac 1{240}-\zeta '(-1)-20\zeta '(-3).
\end{align*}
\end{example}

\begin{prop}\label{spv2}
For $l=0,1$ and $n\in \n$,
\begin{eqnarray*}
\Phi _2^{\rm Rev} (1+l+\rho _n,-l;1,\Lambda )&=&-(\log 2\pi ) \zeta (1+\rho _n)+\frac {\pi l! \od(\rho _n)}{(\rho _n)_{l+1} \sin \left( \pi \rho _n\right)}\\
&+&\sum _{k=1 \atop k:\rm odd}^l \binom l k M(-k)\zeta (1+k+\rho _n).
%&-&\sum _{k=1\atop k:{\rm even}}^l \binom l k \zeta (1+k+\rho _n)(-a_k +k!b_k).
\end{eqnarray*}
\end{prop}
\begin{proof}
We take the limit $s_1\to 1+l+\rho _n$ and then $s_2\to -l$ in (\ref{brief}) with $N=2$.
%sufficiently large $N$. 
Then $I$ vanishes and
we obtain 
\begin{eqnarray*}
\sum _{\rho}&=&(-1)^l l! \od(\rho _n)\Gamma (-l-\rho _n)\Gamma(\rho _n)\\
&=&-\frac {l! \od(\rho _n)}{(\rho _n)_{l+1}}\Gamma (1-\rho _n)\Gamma(\rho _n)\\
&=&-\frac {\pi l! \od(\rho _n)}{(\rho _n)_{l+1} \sin \left( \pi \rho _n\right)}
\end{eqnarray*}
by the known formula
$\Gamma (s)\Gamma (1-s)=\pi/\sin \left( \pi s\right).$
\end{proof}
\if0
\begin{remark}\label{Example-4-2}\ From Proposition \ref{spv2}, we explicitly obtain
\begin{align*}
\Phi _2^{\rm Rev} (1+\rho _n,0;1,\Lambda )&=-(\log 2\pi )\zeta (1+\rho _n)+\frac {\pi \od(\rho _n)} {\rho _n\sin (\pi \rho _n)},\\
\Phi _2^{\rm Rev} (2+\rho _n,-1;1,\Lambda )&=-(\log 2\pi )\zeta (1+\rho _n)+\frac {\pi \od(\rho _n)} {\rho_n(\rho _n+1)\sin (\pi \rho _n)}\\
&+12\zeta '(-1)\zeta (2+\rho _n).
\end{align*}
\end{remark}
\fi

\begin{remark}\label{Rem-5-5} 
We can also compute the regular values of $\Phi _2 (s_1,s_2;1,\Lambda )$ at $(s_1,s_2)=(u_1,u_2)$ on singular sets defined by 
\begin{align*}
& \Phi _2^{\rm Reg} (u_1,u_2;1,\Lambda )=\lim_{s_1\to u_1}\lim_{s_2\to u_2}\Phi _2 (s_1,s_2;1,\Lambda ).
\end{align*}
However it is almost trivial.  In fact, if we consider this limit 
%$$\lim_{s_1\to u_1}\lim_{s_2\to u_2}\,\Phi_2(s_1,s_2;1,\Lambda)$$
in \eqref{brief} for $(u_1,u_2)=(-m,-n)$ and $(k+1+\rho_l,-k)$ $(m,n,k\in \mathbb{N}_0,\,l\in \mathbb{N})$, 
%then only the first and the second terms on the right-hand side of \eqref{brief} remain and the other terms vanish. For example, we can trivially compute
%$$\Phi _2^{\rm Reg} (-m,-n;1,\Lambda )=\frac{\zeta(-m-n-1)}{-n-1}-(\log 2\pi)\zeta(-m-n),$$
%and 
%$$\Phi _2^{\rm Reg} (k+1+\rho_l,-k;1,\Lambda )=-(\log 2\pi)\zeta(1+\rho_l)$$
%for $m,n,k\in \mathbb{N}_0$ and \,$l\in \mathbb{N}$.
then the fifth and the sixth terms simply vanish, 
and the values of other terms can be computed just by substituting the limit values except for the cases which are not convergent as noted in Remark \ref{R-5-2}.
\end{remark}

%%%%%%%%%%%%%%%%%%%%%%%%%%%%%%%%%%%%%%%%%%%%%%%%%%%%%%%%%%%%%
\section{Reverse values of $\Phi _2 (s_1,s_2;1,\mu )$}\label{sec-6}
%%%%%%%%%%%%%%%%%%%%%%%%%%%%%%%%%%%%%%%%%%%%%%%%%%%%%%%%%%%%%

As stated in \eqref{mu-Dir}, we have
$$\frac{1}{\zeta(s)}=\sum _{n=1}^\infty \frac {\mu(n)}{n^s}\quad (\Re s>1).$$
Hence, when $\Phi(s;\alpha)=1$ in \eqref{def_double_tilde}, we see that $\widetilde{\alpha}=\mu$ in this case. Note that $\Phi(s;\alpha)=1$ has no pole. 

Similarly, it is also known that
\begin{align*}
& \frac{\zeta(s-1)}{\zeta(s)}=\sum_{n=1}^\infty \frac {\phi(n)}{n^s}\quad (\Re s>2),\\
& \frac{\zeta(2s)}{\zeta(s)}=\sum_{n=1}^\infty \frac {\lambda(n)}{n^s}\quad (\Re s>1),
\end{align*}
where $\phi$ is the Euler totient function and $\lambda$ is defined by $\lambda(n)=(-1)^r$, where $r$ is the number of prime factors with multiplicity 
(see \cite[$\S$ 1.2]{Titch}). Therefore we can also apply Theorem \ref{Th-4-2} to these cases.

Here we consider the case $\Phi(s;\alpha)=1$, namely $\widetilde{\alpha}=\mu$ and $c_k=c_k(\alpha)$ defined by \eqref{def-ck}. We can see that
\begin{equation}
c_k=\lim_{s\to -k}\ \frac{d}{ds}\frac{s+k}{\zeta(s)}=-\frac{\zeta''(-k)}{2(\zeta'(-k))^2}\quad (\text{$k$:even}). \label{def-ck-2}
\end{equation}

From \eqref{mainthm2}
it follows that
\begin{align}
&\Phi _2(s_1,s_2;1,\mu )\notag \\
&=-2\zeta (s_1+s_2)-\sum _{k=1\atop k {\rm :odd}}^{N-1}\binom {-s_2}k \frac {(k+1)\zeta (s_1+s_2+k)}{B_{k+1}}\nonumber \\
&+\sum _{k=1\atop k {\rm :even}}^{N-1}\bigg[\binom {-s_2}k\bigg\{ \frac {(-1)^{k/2}2(2\pi )^{k}}{\zeta(1+k)}\left(b_k \zeta(s_1+s_2+k)-\frac{\zeta'(s_1+s_2+k)}{k!}\right)\notag\\
& \quad  +c_k\zeta(s_1+s_2+k)\notag\bigg\} -\frac{\Gamma'(s_2+k)}{\Gamma(s_2)}\frac {(-1)^{k/2}2(2\pi )^{k}}{(k!)^2 \zeta (1+k)}\zeta (s_1+s_2+k)\bigg]\nonumber \\
&+\frac 1 {\Gamma (s_2)}\sum _{n=1}^\infty \Gamma (s_2-\rho _n)\Gamma (\rho _n)\frac {\zeta (s_1+s_2-\rho _n)}{\zeta '(\rho _n)}\nonumber \\
&+\frac 1 {2\pi i\Gamma (s_2)}\int _{(N-\varepsilon )}\Gamma (s_2+z)\Gamma (-z)\frac {\zeta (s_1+s_2+z)}{\zeta (-z)}dz\label{mainthm2-2}
\end{align}
for $N\in \mathbb N$, where $\varepsilon$ is a small positive number, $b_k$ and $c_k$ are defined by \eqref{def-bk} and \eqref{def-ck-2}. Therefore 
we can calculate the reverse value of $\Phi _2(s_1,s_2;1,\mu )$ at $(u_1,u_2)$ on singular sets determined by \eqref{coro-2}: 
\begin{align*}
& \Phi _2^{\rm Rev} (u_1,u_2;1,\mu )=\lim_{s_2\to u_2}\lim_{s_1\to u_1}\Phi _2 (s_1,s_2;1,\mu ),
\end{align*}
using the same method as in the proofs of Propositions \ref{spv1} and \ref{spv2}. As noted above, since $\Phi(s;\alpha)=1$ has no pole, \eqref{coro-2} implies
\begin{equation*}
\begin{split}
& s_2=-k\quad (k\in \mathbb{N},\ k\geq 2)\\
& s_1+s_2=1-k\quad (k\in \mathbb{N}_0),\\
& s_2=-l+\rho_n\quad (l\in \mathbb{N}_0,\ n\in \mathbb{N}),\\
& s_1+ s_2=1+\rho_n\quad (n\in \mathbb{N}).
\end{split}
\end{equation*}
For example, we can obtain the following from \eqref{mainthm2-2}. %with $\widetilde{\alpha}=\mu$ and $\Phi(s;\alpha)=1$.

\begin{example}\label{Exam-4-3}\ Under Assumption \ref{Ass-2}, 
\begin{align*}
&\Phi _2^{\rm Rev}(0,0;1,\mu )=13,\\
&\Phi _2^{\rm Rev}(-1,0;1,\mu )=\frac 1 6 -\frac {2\pi ^2}{\zeta (3)}\left(\frac{3}{2}-\gamma\right)-\frac{\zeta''(-2)}{4\zeta'(-2)^2},\\
\if0
&\Phi _2^{\rm Rev}(0,0;1,\mu )=13,\\
&\Phi _2^{\rm Rev}(-1,0;1,\mu )=\frac 1 6 -\frac {4\pi ^2b_2}{\zeta (3)}+\frac {c_2} 2,\\
\fi
&\Phi _2^{\rm Rev}(1+\rho _n,0;1,\mu )=-2\zeta (1+\rho _n)-\frac {\pi } {\rho _n\zeta '(\rho _n)\sin (\pi \rho _n)}, \\
&\Phi _2^{\rm Rev}(2+\rho _n,-1;1,\mu )\\
& \quad =-2\zeta (1+\rho _n) -\frac {\pi } {\rho_n(\rho _n+1)\zeta '(\rho _n)\sin (\pi \rho _n)}-12\zeta (2+\rho _n).
\end{align*}
\end{example}

\section*{Acknowledgments.}
The authors are sincerely grateful to Professors Hirotaka Akatsuka, Hideaki Ishikawa and Masatoshi Suzuki for their useful advice.

\end{document}